\documentclass[11pt]{amsart}
\usepackage{amsmath}
\usepackage{amssymb}
\usepackage{amsthm}
\usepackage{latexsym}
\usepackage{graphicx}
\usepackage{hyperref}
\usepackage{enumerate}
\usepackage[all]{xy}

\newtheorem{thm}{Theorem}[section]

\newtheorem{lem}[thm]{Lemma}
\newtheorem{prop}[thm]{Proposition}
\newtheorem{thmintro}{Theorem}

\theoremstyle{definition}

\newtheorem{cond}[thmintro]{Conditions}

\renewcommand{\thethmintro}{\Alph{thmintro}}

\newcommand{\N}{\mathbb N}
\newcommand{\Z}{\mathbb Z}

\newcommand{\R}{\mathbb R}
\newcommand{\C}{\mathbb C}
\newcommand{\mf}{\mathfrak}
\newcommand{\mc}{\mathcal}

\newcommand{\mr}{\mathrm}
\newcommand{\enuma}[1]{\begin{enumerate}[\textup{(}a\textup{)}] {#1} \end{enumerate}}

\newcommand{\End}{\mathrm{End}}

\newcommand{\matje}[4]{\left(\begin{smallmatrix} #1 & #2 \\ 
#3 & #4 \end{smallmatrix}\right)}

\begin{document}

\title{A comparison of Hochschild homology in algebraic and smooth settings}

\author{David Kazhdan}
\address{Einstein Institute of Mathematics\\
The Hebrew University of Jerusalem\\
Givat Ram, Jerusalem, 9190401, Israel}
\email{kazhdan@math.huji.ac.il} 
\author{Maarten Solleveld}
\address{Institute for Mathematics, Astrophysics and Particle Physics\\
Radboud Universiteit, Heyendaalseweg 135\\
6525AJ Nijmegen, the Netherlands}
\email{m.solleveld@science.ru.nl} 
\date{\today}
\subjclass[2010]{Primary 13D07, 13J10; Secondary 16E40}

\maketitle

\begin{abstract}
Consider a complex affine variety $\tilde V$ and a real analytic Zariski-dense
submanifold $V$ of $\tilde V$. We compare modules over the ring $\mc O (\tilde V)$ of
regular functions on $\tilde V$ with modules over the ring $C^\infty (V)$ of smooth complex
valued functions on~$V$. 

Under a mild condition on the tangent spaces, we prove that $C^\infty (V)$ is flat as a module 
over $\mc O (\tilde V)$. From this we deduce a comparison theorem for the Hochschild homology of 
finite type algebras over $\mc O (\tilde V)$ and the Hochschild homology of similar algebras over 
$C^\infty (V)$. 

We also establish versions of these results for functions on $\tilde V$ (resp. $V$) that are 
invariant under the action of a finite group $G$. As an auxiliary result, we show that 
$C^\infty (V)$ has finite rank as module over $C^\infty (V)^G$.
\end{abstract}

\tableofcontents

\section*{Introduction}

\renewcommand{\thethmintro}{\Alph{thmintro}}

Let $\tilde V$ be a complex affine variety and let $V \subset \tilde V$ be a smooth submanifold.
The general goal of this paper is to compare modules over the algebra of regular functions
$\mc O (\tilde V)$ with modules over the algebra of (complex-valued) smooth functions $C^\infty (V)$.
On the algebraic side $\tilde V$ may be singular. On the smooth side we allow minor singularities
via a smooth action of a finite group $G$, so that we actually consider smooth functions on an
orbifold $V / G$. The precise conditions needed for our of results are:

\begin{cond}\label{cond}
\begin{itemize}
\item[(i)] $V$ is a real analytic Zariski-dense submanifold of $\tilde V$,
\item[(ii)] the action of $G$ on $V$ extends to an action of $G$ on $\tilde V$, 
by algebraic automorphisms,
\item[(iii)] for all $v \in V$, $T_v (\tilde V) = T_v (V) \otimes_\R \C$.
\end{itemize}
\end{cond}
\noindent Typical examples come from real forms of $\tilde V$ (but maybe not all real forms qualify). 
Sometimes (iii) can be replaced by
\begin{itemize}
\item[(iii')] $G$ acts freely on $V$ (e.g. $G =1$) and for each $v \in V$, the real vector 
space $T_v (V)$ spans the complex vector space $T_v (\tilde V)$.
\end{itemize}
The assumptions (i) and (ii) guarantee that $\mc O (\tilde V)$ embeds $G$-equivariantly in 
$C^\infty (V)$. Either of (iii) and (iii') entails that at every point of $V$ the formal completion 
of $\mc O (\tilde V)$ is a subalgebra of the formal completion of $C^\infty (V)$. Under condition
(iii'), $V / G$ can be endowed with the structure of a smooth manifold.

\begin{thmintro}\label{thm:B} 
\textup{(see Theorem \ref{thm:1.1})} \\
Assume that (i), (ii) and (iii) or (iii') hold. Then $C^\infty (V)^G$ is flat over $\mc O (\tilde V)^G$.
\end{thmintro}
The proof runs mainly via formal completions of $C^\infty (V)^G$-modules. We remark that $C^\infty (V)^G$
can be substantially more complicated than $C^\infty (V)$, for instance its Hochschild homology can be
nontrivial in degrees beyond the dimension of $V$.

Our main application of this result is to the Hochschild homology of finite type algebras, as studied
in \cite{KNS}. Recall that a unital algebra $A$ (not necessarily commutative) is a finite type 
$\mc O (\tilde V)^G$-algebra if an algebra homomorphism from $\mc O (\tilde V)^G$ to the centre of $A$ is
given, and makes $A$ into a finitely generated $\mc O (\tilde V)^G$-module. Under the above conditions 
\[
C^\infty (V)^G \underset{\mc O (\tilde V)^G}{\otimes} A
\] 
is a Fr\'echet algebra (this is why we need $V$ to be real-analytic). Furthermore it is finitely 
generated as $C^\infty (V)^G$-module, so it is reasonable to regard it as a smooth finite type algebra.

\begin{thmintro}\label{thm:C} 
\textup{(see Theorem \ref{thm:1.10})} \\
Let $A$ be a unital finite type $\mc O (\tilde V)^G$-algebra and let $M$ be a finitely generated 
$A$-bimodule. Assume that (i), (ii) and (iii) from Conditions \ref{cond} hold. There is 
a natural isomorphism of Fr\'echet $C^\infty (V)^G$-modules
\[
C^\infty (V)^G \underset{\mc O (\tilde V)^G}{\otimes} H_n (A, M) \longrightarrow 
H_n \Big( C^\infty (V)^G \underset{\mc O (\tilde V)^G}{\otimes} A , 
C^\infty (V)^G \underset{\mc O (\tilde V)^G}{\otimes} M \Big) .
\]
\end{thmintro}

We note that on the right hand side the Hochschild homology involves the to\-po\-lo\-gy of the algebra,
via the complete projective tensor product of Fr\'echet spaces. Theorem \ref{thm:C} is a smooth 
version of an earlier result with formal completions \cite[Theorem 3]{KNS}. An advantage of Theorem 
\ref{thm:C} is that it reduces the computation of the Hochschild homology of certain Fr\'echet 
algebras to the Hochschild homology of finite type algebras, about which a lot is known from \cite{KNS}.

It would be interesting to draw consequences from Theorem \ref{thm:C} for the periodic 
cyclic homology of $A$ \cite[\S 5.1.3]{Lod} and of the Fr\'echet algebra 
$C^\infty (V)^G \underset{\mc O (\tilde V)^G}{\otimes} A$ (using the definition from 
\cite[\S 2]{BrPl}). The periodic cyclic homology 
$HP_n (A)$ of the finite type algebra $A$ was analysed in \cite[\S 4]{KNS}. For instance, it 
follows easily from \cite[Theorem 10]{KNS} that $HP_n (A)$ has finite dimension. The inclusion
\[
A \to C^\infty (V)^G \underset{\mc O (\tilde V)^G}{\otimes} A \quad \text{induces a linear map}
\quad HP_n (A) \to HP_n \big( C^\infty (V)^G \underset{\mc O (\tilde V)^G}{\otimes} A \big) ,
\]
which in some cases is a bijection \cite[\S 1]{Sol}. However, that does not hold in the 
generality of Theorem \ref{thm:C}, for instance because $V$ and $\tilde V$ 
can have different cohomology.

To improve the effect of Theorem \ref{thm:C} for the computation of the Hochschild homology of 
smooth finite type algebras, we make its left hand side explicit in some cases.
That involves one result of general nature:

\begin{thmintro}\label{thm:D} 
\textup{(see Theorem \ref{thm:1.5})} \\
Let $V$ be a smooth manifold with an action of a finite group $G$. Then
$C^\infty (V)$ is finitely generated as $C^\infty (V)^G$-module.
\end{thmintro}

Let $\Omega^n (\tilde V)$ be the $\mc O (\tilde V)$-module of algebraic differential $n$-forms
on $\tilde V$ and denote the $C^\infty (V)$-module of smooth $n$-forms on $V$ by $\Omega^n_{sm}(V)$.

\begin{thmintro}\label{thm:E} 
\textup{(a special case of Lemma \ref{lem:1.9})} \\
Suppose that (i), (ii) and (iii) from Conditions \ref{cond} hold.
There is a natural isomorphism of Fr\'echet $C^\infty (V)^G$-modules
\[
C^\infty (V)^G \underset{\mc O (\tilde V)^G}{\otimes} \Omega^n (\tilde V)
\;\cong\; \Omega^n_{sm} (V) .
\]
\end{thmintro}

\noindent From Theorems \ref{thm:C} and \ref{thm:E} one can easily deduce a smooth version of
the Hochschild--Kostant--Rosenberg theorem, see Section \ref{sec:ex}. Obviously that would be 
an extremely roundabout proof. The advantage of our methods is rather that they apply to much 
wider classes of algebras, possibly noncommutative. In particular our results will be useful for 
the computation of the Hochschild homology of the Harish-Chandra--Schwartz algebra of a 
reductive $p$-adic group, for which we refer to \cite{Sol2}.\\[1mm]

\textbf{Acknowledgements.}\\
We thank Roman Bezrukavnikov and Sacha Braverman for interesting discussions,
which motivated these investigations.

A big thanks goes to the referees, whose detailed reports helped to solve a lot of problems
and lead to big improvements in the paper.

\renewcommand{\theequation}{\arabic{section}.\arabic{equation}}
\numberwithin{equation}{section}

\section{Flatness of smooth functions as module over regular functions}

Let $V$ be a smooth manifold (without boundary) and let $G$ be a finite group acting on $V$ 
by diffeomorphisms. Consider the algebra $C^\infty (V)^G$ of $G$-invariant smooth complex-valued 
functions on $V$. For each $v \in V$ we have the closed maximal ideal $I_v \subset C^\infty (V)$ 
of functions vanishing at $v$ and the closed ideal $I_{Gv}$ of functions vanishing on $Gv$.
The $G$-invariant elements in the latter form an ideal $I_{Gv}^G \subset C^\infty (V)^G$. Let 
$FP_v$ be the Fr\'echet algebra of formal power series on an in\-fi\-ni\-tesimal neighborhood of $v$ in 
$V$ and let $FP_v^{G_v}$ be the subalgebra of $G_v$-invariants. Then
$FP_v \cong \varprojlim\nolimits_n C^\infty (V) / I_v^n$  and
\begin{equation}\label{eq:2.1}
FP_v^{G_v} \cong \big( \bigoplus\nolimits_{v' \in Gv} FP_{v'} \big)^G \cong 
\big( \varprojlim\nolimits_n C^\infty (V) / I_{Gv}^n \big)^G \cong
\varprojlim\nolimits_n C^\infty (V)^G \big/ I_{Gv}^{n,G} .
\end{equation}
By a theorem of Borel (see \cite[Th\'eor\`eme IV.3.1 and Remarque IV.3.5]{Tou} or
\cite[Theorem 26.29]{MeVo}) the Taylor series map
\[
\mf T_v : C^\infty (M) \to FP_v
\]
is surjective. Its kernel is the module $I_v^\infty$ of functions that are flat at $v$. 
Similarly we have the ideal 
\[
I_{Gv}^\infty = \bigcap\nolimits_{v' \in Gv} I_{v'}^\infty \: \subset \: C^\infty (V)
\]
of functions that are flat on $Gv$. In these terms \eqref{eq:2.1} becomes an isomorphism
\begin{equation}\label{eq:1.4}
FP_v^{G_v} \cong C^\infty (V)^G \big/ I_{Gv}^{\infty,G}. 
\end{equation}
For any Fr\'echet $C^\infty (V)^G$-module $M$ we can form the ``formal completion" at $v$:
\begin{equation}\label{eq:1.1}
\hat M_{Gv} := FP_v^{G_v} \underset{C^\infty (V)^G}{\hat{\otimes}} M \cong 
M \big/ \overline{I_{Gv}^{\infty,G} M} .
\end{equation}
In contrast with the algebraic setting, $\hat M_{Gv}$ is actually a 
quotient rather than a completion of $M$.

\begin{lem}\label{lem:1.3}
Let $M$ be a finitely generated Fr\'echet $C^\infty (V)^G$-module. Let $M_1$ and $M_2$ be 
closed $C^\infty (V)^G$-submodules of $M$, such that $M_1 \supset M_2$ and 
$\widehat{M_1}_{Gv} = \widehat{M_2}_{Gv}$ for all $v \in V$. Then $M_1 = M_2$. 
\end{lem}
\begin{proof}
By assumption there exists a finitely generated free $C^\infty (V)^G$-module $N$ and a surjective
homomorphism of Fr\'echet $C^\infty (V)^G$-modules $p : N \to M$. By the continuity of $p$, 
$N_i := p^{-1}(M_i)$ is a closed $C^\infty (V)^G$-submodule of $N$. For any $v \in V$ we have
\[
\widehat{N_1 / N_2}_{Gv} \cong \widehat{M_1 / M_2}_{Gv} = 0.
\]
From that and \eqref{eq:1.1} we deduce
\begin{equation}\label{eq:1.2}
N_1 / N_2 = \overline{I_{Gv}^{\infty,G} (N_1 / N_2)} = \overline{I_{Gv}^{\infty,G} N_1 + N_2} / N_2 
\subset \overline{I_{Gv}^{\infty,G} N + N_2} / N_2 .
\end{equation}
This holds for all $v \in V$, so 
\[
N_1 \subset \bigcap\nolimits_{v \in V} \overline{I_{Gv}^{\infty,G} N + N_2} .
\]
Consider the finitely generated free $C^\infty (V)$-module 
$C^\infty (V) \underset{C^\infty (V)^G}{\hat{\otimes}} N$. In there we have $C^\infty (V)$-submodules
\[
C^\infty (V) N_1 \: \subset \: C^\infty (V) 
\Big( \bigcap\nolimits_{v \in V} \overline{I_{Gv}^{\infty,G} N + N_2} \Big) \: \subset \:
\bigcap\nolimits_{v \in V} \overline{  I_{Gv}^\infty N + C^\infty (V) N_2 } .
\]
Applying the Taylor series map, we find 
\[
\mf T_v (C^\infty (V) N_1) \subset \overline{\mf T_v (C^\infty (V) N_2)} .
\]
By \cite[Th\'eor\`eme V.1.3]{Tou} for the variety $\{v\}$, the right hand side equals 
$\mf T_v (C^\infty (V) N_2)$. As $N_1 \supset N_2$, we deduce that $C^\infty (V) N_1$ and 
$C^\infty (V) N_2$ have the same Taylor series at every $v \in V$.
By Whitney's spectral theorem \cite[Corollaire V.1.6]{Tou}, this implies 
\begin{equation}\label{eq:2.2}
C^\infty (V) N_1 \subset \overline{C^\infty (V) N_2} .
\end{equation}
Taking $G$-invariants inside $C^\infty (V) \underset{C^\infty (V)^G}{\hat{\otimes}} N$, we obtain
\[
N_1 = (C^\infty (V) N_1 )^G \subset \overline{C^\infty (V) N_2}^G = \overline{N_2} = N_2.
\]
Hence $N_1 = N_2$ and $M_1 = M_2$.
\end{proof}

In this context it useful to mention the following slight generalization of a result of Malgrange 
\cite[Corollaire VI.1.8]{Tou}.

\begin{thm}\label{thm:1.7}
Assume that $V$ is real analytic, and let $M$ be a $C^\infty (V)^G$-submodule of 
$\big( C^\infty (V)^G \big)^r$ generated by finitely many real-analytic $G$-invariant functions 
from $V$ to $\C^r$. Then $M$ is closed in $\big( C^\infty (V)^G \big)^r$.
\end{thm}
\begin{proof} 
Let $\{f_i \}$ be a finite set of analytic $G$-invariant functions from $V$ to $\C^r$. By 
\cite[Corollaire VI.1.8]{Tou} they generate a closed $C^\infty (V)$-submodule $M'$ of $C^\infty (V)^r$.

Assume that the $f_i$ generate $M$ as $C^\infty (V)^G$-module. Write $p_G = |G|^{-1} \sum_{g \in G} g$,
 an idempotent in $\C [G]$. Clearly $M \subset M' \cap (C^\infty (V)^G)^r$. On the other hand
\begin{align*}
M = \sum\nolimits_i C^\infty (V)^G f_i & = \sum\nolimits_i \big( p_G C^\infty (V) \big) f_i = 
p_G \sum\nolimits_i \big( p_G C^\infty (V) f_i \big) \\
& = p_G \big( \sum\nolimits_i C^\infty (V) f_i \big) = p_G M' \supset M' \cap (C^\infty (V)^G)^r.
\end{align*}
Hence $M = M' \cap (C^\infty (V)^G)^r$, which is closed in $(C^\infty (V)^G)^r$ because $M'$ is 
closed in $C^\infty (V)^r$.
\end{proof}

Let $\tilde V$ be a complex affine $G$-variety and recall the Conditions \ref{cond}.

\begin{lem}\label{lem:1.8}
Assume (i) and (ii) from Conditions \ref{cond} and let $M$ be a finitely ge\-ne\-ra\-ted 
$\mc O (\tilde V)^G$-module. The $C^\infty (V)^G$-modules 
\[
C^\infty (V)^G \underset{\mc O (\tilde V)^G}{\otimes} M ,\qquad
FP_v^{G_v} \underset{\mc O (\tilde V)^G}{\otimes} M \qquad \text{and} \qquad 
I_{Gv}^{\infty,G} \big( C^\infty (V)^G \underset{\mc O (\tilde V)^G}{\otimes} M \big)
\]
are nuclear Fr\'echet. The first two are generated by a finite subset of $M$.
\end{lem}
\begin{proof}
Any finite set of generators of $M$ as $\mc O (\tilde V)$-module also generates the first two
$C^\infty (V)^G$-modules under consideration. By \cite[(30) and subsequent lines]{OpSo}, every
finitely generated $FP_v^{G_v}$-module is Fr\'echet, so in particular 
$FP_v^{G_v} \underset{\mc O (\tilde V)^G}{\otimes} M$.\\
Pick $r \in \Z_{>0}$ and a $\mc O (\tilde V)^G$-submodule $N$ of $\big( \mc O (\tilde V)^G 
\big)^r$ such that $M \cong \big( \mc O (\tilde V)^G \big)^r / N$. 
The kernel of the surjective homomorphism of $C^\infty (V)^G$-modules
\begin{equation}\label{eq:1.10}
\begin{aligned}
\big( C^\infty (V)^G \big)^r = \; & C^\infty (V)^G \underset{\mc O (\tilde V)^G}{\otimes} 
\big( \mc O ( \tilde V)^G \big)^r \longrightarrow \\ 
& C^\infty (V)^G \underset{\mc O (\tilde V)^G}{\otimes} \big( \mc O ( \tilde V)^G \big)^r / N 
= C^\infty (V)^G \underset{\mc O (\tilde V)^G}{\otimes} M
\end{aligned}
\end{equation}
is generated by $1 \otimes N$. Since $\mc O (\tilde V)^G$ is Noetherian, $N$ is generated as 
$\mc O (\tilde V)^G$-module by some finite subset $S_N$. Then the kernel of \eqref{eq:1.10} 
is generated by $1 \otimes S_N$. The analyticity of $V$ entails that $S_N$ consists of analytic 
$G$-invariant functions from $V$ to $\C^r$. Now Theorem \ref{thm:1.7} says that the kernel of 
\eqref{eq:1.10} is closed in $\big( C^\infty (V)^G \big)^r$. Hence $C^\infty (V)^G 
\underset{\mc O (\tilde V)^G}{\otimes} M$ is the quotient of $\big( C^\infty (V)^G \big)^r$ 
by a closed subspace, and in particular is a Fr\'echet space. 

In the short exact sequence of topological vector spaces
\[
0 \to I_{Gv}^{\infty,G} \big( C^\infty (V)^G \underset{\mc O (\tilde V)^G}{\otimes} M \big) \to
C^\infty (V)^G \underset{\mc O (\tilde V)^G}{\otimes} M \to
FP_v^{G_v} \underset{\mc O (\tilde V)^G}{\otimes} M \to 0
\]
the middle term is Fr\'echet and the right hand side is Hausdorff. Hence the left hand side
is a closed subspace of the middle term, and is itself Fr\'echet.

Next we address the nuclearity. Our arguments are based entirely on the in\-he\-ri\-tance 
properties for nuclearity, which can be found for instance in \cite[Satz 28.6--28.7]{MeVo} and
\cite[Theorem 7.4]{ScWo}. The power series ring $FP_v$ is a direct product of copies of $\C$,
so it is nuclear. Then its subspace $FP_v^{G_v}$ and the finite direct sum $(F_v^{G_v})^r$
with $r \in \N$ inherit nuclearity from $FP_v$. As $FP_v^{G_v} \underset{\mc O 
(\tilde V)^G}{\otimes} M$ is a Hausdorff quotient of $(F_v^{G_v})^r$ for a suitable $r$,
it is nuclear as well.

The Fr\'echet space $C^\infty (V)$ is a standard example of a nuclear space \cite[p. 108]{ScWo}.
Hence the subspace $C^\infty (V)^G$ and $(C^\infty (V)^G)^r$ are also nuclear. We showed that
$C^\infty (V)^G \underset{\mc O (\tilde V)^G}{\otimes} M$ is a Hausdorff quotient of 
$(C^\infty (V)^G)^r$, and therefore nuclear. Finally, nuclearity is inherited by the subspace
$I_{Gv}^{\infty,G} \big( C^\infty (V)^G \underset{\mc O (\tilde V)^G}{\otimes} M \big)$.
\end{proof}

From Lemma \ref{lem:1.8} we obtain a functor 
\begin{equation}\label{eq:1.3}
C^\infty (V)^G \underset{\mc O (\tilde V)^G}{\otimes} :
\mr{Mod}_{fg} \big( \mc O (\tilde V)^G \big) \to \mr{Mod}_{Fr} \big( C^\infty (V)^G \big) ,
\end{equation}
where the subscripts fg and Fr stand for finitely generated and Fr\'echet, respectively.

\begin{lem}\label{lem:1.11}
Assume that (i), (ii) and either (iii) or (iii') from Conditions \ref{cond} hold and let 
$M \subset M'$ be finitely generated $\mc O (V)^G$-modules. Then the natural map
\[
\big( C^\infty (V)^G \underset{\mc O (\tilde V)^G}{\otimes} M \big)^\wedge_{Gv} \longrightarrow
\big( C^\infty (V)^G \underset{\mc O (\tilde V)^G}{\otimes} M' \big)^\wedge_{Gv} 
\]
is injective.
\end{lem}
\begin{proof}
Recall that the formal completion of the $\mc O (\tilde V)^G$-module $M$ at $Gv \in V /G$ 
is defined as
\begin{equation}\label{eq:1.9}
\hat M_{Gv} = \varprojlim\nolimits_n M / \big( I_{Gv}^n \cap \mc O (\tilde V)^G \big) M. 
\end{equation}
Let $\widetilde{FP}_v$ be the formal completion of $\mc O (\tilde V)$ at $v \in V$. Like 
in \eqref{eq:2.1}, $\widetilde{FP}_v^{G_v}$ is the formal completion of $\mc O (\tilde V)^G$ 
at $Gv$, and it can be considered as a subalgebra of $FP_v^{G_v}$. Since $M$ is finitely 
generated, there is a natural isomorphism
\[
\hat M_{Gv} \cong \widetilde{FP}_v^{G_v} \underset{\mc O (\tilde V)^G}{\otimes} M .
\]
There are isomorphisms of $FP_v^{G_v}$-modules
\begin{equation}\label{eq:1.12}
\begin{aligned}
\big( C^\infty (V)^G \underset{\mc O (\tilde V)^G}{\otimes} M \big)^\wedge_{Gv} & = 
FP_v^{G_v} \underset{C^\infty (V)^G}{\hat{\otimes}} C^\infty (V)^G \underset{\mc O 
(\tilde V)^G}{\otimes} M = FP_v^{G_v} \underset{\mc O (\tilde V)^G}{\otimes} M \\
& \cong FP_v^{G_v} \underset{\widetilde{FP}_v^{G_v}}{\otimes} \widetilde{FP}_v^{G_v}  
\underset{\mc O (\tilde V)^G}{\otimes} M \cong
FP_v^{G_v} \underset{\widetilde{FP}_v^{G_v}}{\otimes} \hat M_{Gv} .
\end{aligned}
\end{equation}
By the exactness of the formal completion functor \eqref{eq:1.9} for finitely generated 
$\mc O (\tilde V)^G$-modules,
\begin{equation}\label{eq:1.13}
\hat M_{Gv} \text{ is a } \widetilde{FP}_v^{G_v}\text{-submodule of } \hat{M'}_{Gv}. 
\end{equation}
Suppose that (iii) holds. Then $FP_v \cong \widetilde{FP_v}$ as $G_v$-representations, so 
$FP_v^{G_v} \cong \widetilde{FP}_v^{G_v}$. With the isomorphism \eqref{eq:1.12}
that immediately implies the statement.

Suppose that (iii') holds, so that $G_v = 1$. The canonical surjection
\[
T_v (V) \otimes_\R \C \to T_v (\tilde V)
\]
induces an injection
\[
j : T_v (\tilde V)^* \to ( T_v (V) \otimes_\R \C)^* 
\]
Pick a basis $\{z_1,\ldots,z_d\}$ of $j (T_v (\tilde V)^*)$ and extend it with elements  
$\{w_1,\ldots,w_{\dim V - d} \}$ to a basis of $( T_v (V) \otimes_\R \C)^*$. 
There are isomorphisms of Fr\'echet algebras
\begin{equation}\label{eq:1.11}
FP_v \cong \C [[z_1,\ldots,z_d,w_1,\ldots,w_{\dim V - d}]] \cong 
\widetilde{FP}_v [[w_1,\ldots,w_{\dim V - d}]] .
\end{equation}
Thus the $\widetilde{FP}_v$-module $FP_v$ is isomorphic to a product of copies of
$\widetilde{FP}_v$, indexed by all the monomials built from $\{w_1,\ldots,w_{\dim V - d} \}$.
Furthermore $\widetilde{FP}_v$ is Noetherian, and then \cite[Theorem 2.2]{Cha} says that
$FP_v$ is flat over $\widetilde{FP}_v$. Combine that with \eqref{eq:1.12} and \eqref{eq:1.13}.
\end{proof}

We note that Lemma \ref{lem:1.11} may become false if we assume only (i), (ii) and (iii') in a weaker
version without the freedom of the $G$-action. For example, take $V = \tilde V = \C$, on which 
$G = \{1,-1\}$ acts by multiplication. For $v = 0$ we have 
\[
\widetilde{FP}_v^{G_v} = \C [[z^2]] \quad \text{and} \quad
FP_v^{G_v} = \C[[z,\bar{z}]]^G = \C[[z^2,\bar{z}^2, z \bar{z}]] .
\]
Here $FP_v^{G_v}$ is not flat over $\widetilde{FP}_v^{G_v}$, and then we see from \eqref{eq:1.12}
that Lemma \ref{lem:1.11} fails.\\

The main result of this section generalizes the flatness of $C^\infty (\tilde V)$ over 
$\mc O (\tilde V)$. As pointed out in an answer to a question on 
MathOverflow\footnote{mathoverflow.net/questions/226136/is-the-sheaf-of-smooth-functions-flat} that 
case can be shown quickly with results of Malgrange \cite{Mal} about complex analytic functions.

\begin{thm}\label{thm:1.1}
Assume that (i), (ii) and either (iii) or (iii') from Conditions \ref{cond} hold.
Then $C^\infty (V)^G$ is flat as $\mc O (\tilde V)^G$-module.
In particular the functor \eqref{eq:1.3} is exact.
\end{thm}
\begin{proof}
According to \cite[Proposition 6.1]{Eis}, flatness can be checked by testing it with finitely
generated modules. Let $M \subset M'$ be finitely generated $\mc O (\tilde V)^G$-modules.
We need to show that the natural map
\[
\mu : C^\infty (V)^G \underset{\mc O (\tilde V)^G}{\otimes} M \to 
C^\infty (V)^G \underset{\mc O (\tilde V)^G}{\otimes} M'
\]
is injective. We want to apply Lemma \ref{lem:1.3} inside the domain of $\mu$, which by Lemma 
\ref{lem:1.8} has the right properties. The submodules will be $M_2 = 0$ and $M_1 = \ker (\mu)$, 
which is a closed submodule of the domain because $\mu$ is continuous and $C^\infty (V)^G$-linear. 
Lemma \ref{lem:1.3} yields the desired conclusion $\ker (\mu) = 0$, provided we can check that 
all formal completions of the $C^\infty (V)^G$-module $\ker (\mu)$ are zero.

From Lemma \ref{lem:1.11} we know that $\hat{\mu}_{Gv}$ is injective. We would like to apply 
the exactness of the formal completion functor from \cite[Theorem 2.5]{OpSo} to
\[
0 \to \ker (\mu) \to C^\infty (V)^G \underset{\mc O (\tilde V)^G}{\otimes} M \to 
C^\infty (V)^G \underset{\mc O (\tilde V)^G}{\otimes} M' ,
\]
but unfortunately $\ker (\mu)$ could be a topological vector space of a more general kind than
allowed by \cite[Theorem 2.5]{OpSo}. It turns out that we can still use the proof of
\cite[Theorem 2.5]{OpSo}, which relies on technical constructions in \cite[Chapitre 1]{MeTo}. 
Consider an element of $\widehat{\ker (\mu)}_{Gv}$ represented by $m \in \ker (\mu) \subset
C^\infty (V)^G \otimes M$. By the injectivity of $\hat{\mu}_{Gv}$ and \eqref{eq:1.1}, $m$
belongs to
\[
\overline{I_{Gv}^{\infty,G} \big( C^\infty (V)^G \underset{\mc O (\tilde V)^G}{\otimes} M \big)} .
\]
Here taking the closure is superfluous, for by Lemma \ref{lem:1.8} it is already a closed subspace
of $C^\infty (V)^G \underset{\mc O (\tilde V)^G}{\otimes} M$.
Hence there are finitely many $f_j \in I_{Gv}^{\infty,G}$
and $m_j \in M$ such that $m = \sum_j f_j \otimes m_j$. By \cite[p. 183]{MeTo} there exists 
a $\psi \in I_{Gv}^\infty$ such that 
\[
f_j / \psi \in I_{Gv}^\infty \subset C^\infty (V) \text{ for all } j.
\] 
The construction of $\psi$ in \cite[p. 184]{MeTo} runs via a sequence of functions 
$\epsilon_i \in I_{Gv}^\infty$ such that $\sum_i \epsilon_i = \psi$ and $\epsilon_i / \psi \in 
I_{Gv}^\infty$. We can average all these functions $\epsilon_i$ over $G$, that preserves 
their properties used in \cite{MeTo}. Hence we may assume that all the $\epsilon_i$ are 
$G$-invariant and that $\psi \in I_{Gv}^{\infty,G} \subset C^\infty (V)^G$. Then 
\[
m / \psi = \sum\nolimits_j f_j / \psi \otimes m_j \in 
C^\infty (V)^G \underset{\mc O (\tilde V)^G}{\otimes} M
\] 
is well-defined. By $G$-invariance $(\epsilon_i / \psi) m \in I_{Gv}^{\infty,G} \ker (\mu)$. 
Now we can write
\[
m = \psi \cdot m / \psi = \sum\nolimits_i \epsilon_i \cdot m / \psi = 
\sum\nolimits_i (\epsilon_i / \psi) \cdot m \in \overline{I_{Gv}^{\infty,G} \ker (\mu)} .
\]
The sums converge by the equalities (although $\sum\nolimits_i (\epsilon_i / \psi)$ need not
converge). Hence $m = 0$ in $\widehat{\ker (\mu)}_{Gv}$, and $\widehat{\ker (\mu)}_{Gv} = 0$.
\end{proof}

\section{Finite type algebras and their smooth versions}

We will apply Theorem \ref{thm:1.1} to finite type algebras. By an $\mc O (\tilde V)^G$-algebra
we mean a (not necessarily unital) algebra $A$ together with a unital algebra homomorphism
from $\mc O (\tilde V)^G$ to the centre of the multiplier algebra of $A$. Recall from \cite{KNS}
that $A$ has finite type (as $\mc O (\tilde V)^G$-algebra) if it is finitely generated as
module over $\mc O (\tilde V)^G$. The structure, homology and representation theory of such algebras 
were studied in \cite{KNS}. In particular $A$ is always a polynomial identity algebra.
We want to compare $A$ and
\[
C^\infty (V)^G \underset{\mc O (\tilde V)^G}{\otimes} A .
\]
By Lemma \ref{lem:1.8} this is a Fr\'echet algebra, and it is finitely generated as module
over $C^\infty (V)^G$. It is also a polynomial identity algebra, and  we regard it as a smooth 
version of a finite type algebra.

Assume that $A$ is unital and let $M$ be a finitely generated $A$-module. By \cite[Lemma 3]{KNS} 
it has a resolution $(A \otimes_\C F_*, d_*)$ consisting of finitely generated free $A$-modules.

\begin{lem}\label{lem:1.2}
Assume that (i), (ii) and (iii) or (iii') from Conditions \ref{cond} hold and put
$C_n = C^\infty (V)^G \underset{\mc O (\tilde V)^G}{\otimes} A \underset{\C}{\otimes} F_n$.
\enuma{
\item $(C_n, \mr{id} \otimes d_n)$ is a resolution of the 
\[
C^\infty (V)^G \underset{\mc O (\tilde V)^G}{\otimes} A\text{-module} \quad 
C^\infty (V)^G \underset{\mc O (\tilde V)^G}{\otimes} M 
\]
by finitely generated free modules. 
\item Suppose in addition that $C^\infty (V)^G \underset{\mc O (\tilde V)^G}{\otimes} A$
and $C^\infty (V)^G \underset{\mc O (\tilde V)^G}{\otimes} M$ are isomorphic (as Fr\'echet spaces)
to direct summands of the space of rapidly decreasing sequences $\mc S (\N)$. Then the resolution 
from part (a) is split exact as complex of Fr\'echet spaces.
}
\end{lem}
\begin{proof}
(a) The exactness of 
\begin{equation}\label{eq:2.4}
C_* \to C^\infty (V)^G \underset{\mc O (\tilde V)^G}{\otimes} M
\end{equation}
is a direct consequence of Theorem \ref{thm:1.1}.\\
(b) Let $\mc D$ be the category of Fr\'echet spaces that are isomorphic to direct summands of
$\mc S (\N)$. Recall that $\mc S (\N)^d \cong \mc S (\N)$ for all $d \in \N$. Hence $C_n$ belongs
to $\mc D$ and \eqref{eq:2.4} is an exact sequence in $\mc D$. By \cite[Theorems 1.8 and 5.1]{Vog},
every exact sequence in $\mc D$ admits a continuous linear splitting. 
\end{proof}

\noindent We turn to a comparison of the Hochschild homologies of $A$ and of 
$C^\infty (V)^G \underset{\mc O (\tilde V)^G}{\otimes} A$. For a unital finite type algebra $A$
and an $A$-bimodule $M$, this can be defined as
\[
H_n (A,M) = \mr{Tor}_n^{A \otimes A^{op}} (A,M) ,
\] 
see \cite[Proposition 1.1.13]{Lod}. The special case $M = A$ is by definition the Hochschild homology
$HH_n (A)$. 

For Fr\'echet algebras like 
$C^\infty (V)^G \underset{\mc O (\tilde V)^G}{\otimes} A$, the topology must be taken into 
account. This is done best by fixing a (completed) topological tensor product and building all
differential complexes with respect to this tensor product, see for instance \cite{Tay}.

We do it slightly differently though, with bornologies and bornological modules \cite[\S 2]{Mey1}. This 
approach has the advantage that both $A$ and $C^\infty (V)^G \underset{\mc O (\tilde V)^G}{\otimes} A$ 
can be regarded as complete bornological algebras. For $A$ it boils down to the standard purely 
algebraic setup, while for Fr\'echet algebras/modules the bornological structure is equivalent 
to the topological structure. The appropriate tensor product is the complete bornological tensor product 
$\hat \otimes$, which for Fr\'echet spaces agrees with the complete projective tensor product 
\cite[Theorem I.87]{Mey2}. By default we endow all finitely generated $\mc O (\tilde V)^G$-modules with 
the fine bornology \cite[\S 2.1]{Mey1}, so that complete bornological tensor products also make sense 
for them (and they agree with the algebraic tensor products).

The category of bornological modules of a complete bornological algebra $B$ is made into an exact 
category by allowing only extensions of $B$-modules that are split as extensions of bornological 
vector spaces. For extensions of Fr\'echet $B$-modules, this just means that they must be split 
as extensions of Fr\'echet spaces. It was checked in \cite[\S 3]{Mey1} that this is an excellent 
setting for homological algebra. 

Assume that $B$ is unital, and let $N$ be a bornological $B$-bimodule.
A good definition of the Hochschild homology of $B$ with coefficients in $N$ is
\begin{equation}\label{eq:1.14}
H_n (B,N) = \mr{Tor}_n^{B \hat \otimes B^{op}} (B,N),
\end{equation}
in the exact category of bornological $B$-modules. For $N = B$ this yields the Hochschild homology
$HH_n (B)$. For Fr\'echet algebras and modules, \eqref{eq:1.14} agrees with the definition in terms
of the completed projective tensor product \cite{Tay}.

From \eqref{eq:1.14} we see that we will have to consider some modules over 
\[
C^\infty (V)^G \hat \otimes C^\infty (V)^G \cong C^\infty (V \times V)^{G \times G} .
\] 

\begin{lem}\label{lem:1.12}
Suppose that (i), (ii) and (iii) from Conditions \ref{cond} hold. 
\enuma{
\item $\mc O (\tilde V)^G$ is dense in $C^\infty (V)^G$.
\item For any finitely generated $\mc O (\tilde V)^G$-bimodule $M$, 
there is a natural isomorphism of $C^\infty (V)^G$-modules
\[
C^\infty (V)^G \underset{\mc O (\tilde V)^G}{\otimes} M \to C^\infty (V)^G 
\underset{\mc O (\tilde V)^G}{\hat \otimes} M \underset{\mc O (\tilde V)^G}{\hat \otimes} 
C^\infty (V)^G : f \otimes m \mapsto f \otimes m \otimes 1 .
\]
When $M$ is an $\mc O (\tilde V)^G$-algebra, this map is an algebra isomorphism.
}
\end{lem}
\begin{proof}
(a) It suffices to show that $\mc O (\tilde V)$ is dense in $C^\infty (V)$, because from that we
can obtain the statement by applying the idempotent $p_G$. For any $v \in V$, (iii) yields
a natural isomorphism 
\[
\widehat{\mc O (\tilde V)}_v = \widetilde{FP}_v \cong FP_v = \widehat{C^\infty (V)}_v .
\]
According to \cite[Corollaire V.1.6]{Tou} this implies that the closure of 
$\mc O (\tilde V)$ in $C^\infty (V)$ is $C^\infty (V)$.\\
(b) By Lemma \ref{lem:1.8} (applied to $\tilde V \times \tilde V$ with the $G \times G$-action),
\begin{equation}\label{eq:1.18}
\big( C^\infty (V)^G \hat \otimes C^\infty (V)^G \big) \underset{\mc O (\tilde V)^G \otimes 
\mc O (\tilde V)^G}{\otimes} M = C^\infty (V)^G \underset{\mc O (\tilde V)^G}{\hat \otimes} M
\underset{\mc O (\tilde V)^G}{\hat \otimes} C^\infty (V)^G 
\end{equation}
is a Fr\'echet space. Let $x \in C^\infty (V)^G \underset{\mc O (\tilde V)^G}{\otimes} M$ and
$f \in C^\infty (V)^G$. By part (a) there exists a sequence $(f_n )_{n=1}^\infty$ in 
$\mc O (\tilde V)^G$ converging to $f$. The space \eqref{eq:1.18} is Hausdorff, so limits
are unique in there and we can compute
\[
x \otimes f = \lim_{n \to \infty} x \otimes f_n = 
\lim_{n \to \infty} x f_n \otimes 1 = x f \otimes 1 .
\]
Consequently \eqref{eq:1.18} equals 
\begin{equation}\label{eq:1.19}
C^\infty (V)^G \underset{\mc O (\tilde V)^G}{\hat \otimes} M \underset{C^\infty (V)^G}{\hat 
\otimes} C^\infty (V)^G = C^\infty (V)^G \underset{\mc O (\tilde V)^G}{\hat \otimes} M .
\end{equation}
Since $C^\infty (V)^G \underset{\mc O (\tilde V)^G}{\otimes} M$ already is Fr\'echet (by
Lemma \ref{lem:1.8}), it equals the right hand side of \eqref{eq:1.19}. It is easy to see
that this isomorphism of Fr\'echet $C^\infty (V)^G$-modules is given by the map in the statement.
 
When $M$ is an addition an $\mc O (\tilde V)^G$-algebra, the map in the statement is also an 
algebra homomorphism, so in fact an algebra isomorphism.
\end{proof}

Lemmas \ref{lem:1.2} and \ref{lem:1.12} together say that, under the topological condition from
Lemma \ref{lem:1.2}.b, the embedding of bornological algebras 
\[
A \to C^\infty (V)^G \underset{\mc O (\tilde V)^G}{\otimes} A
\] is a homological epimorphism. That implies several 
comparison results for homological properties of the derived module categories of the two involved 
algebras, see \cite[Theorem 35]{Mey1} (where this is called an isocohomological embedding). 

Since the Fr\'echet space $C^\infty (V)^G$ is isomorphic to a direct summand of 
$\mc S (\N)$ when $V$ is compact \cite[Satz 31.16]{MeVo}, it seems likely that in many cases 
$C^\infty (V)^G \underset{\mc O (\tilde V)^G}{\otimes} A$ has the same property. Proving that
is another matter though. Fortunately, we can work around the existence of continuous linear 
splittings of our resolutions by involving properties of nuclear Fr\'echet spaces.

One can compute $H_n (B,N)$ (at least when $B$ is unital) with a completed version of the standard
bar-resolution of $B$ \cite[\S 1]{Lod}, but the definition as a derived functor is more flexible.
The inclusion $A \to C^\infty (V)^G \underset{\mc O (\tilde V)^G}{\otimes} A$ induces a chain
map between the respective bar-resolutions, and hence induces a natural map
\begin{equation}\label{eq:1.15}
H_n (A,M) \to H_n \Big( C^\infty (V)^G \underset{\mc O (\tilde V)^G}{\otimes} A , 
C^\infty (V)^G \underset{\mc O (\tilde V)^G}{\otimes} M \Big) .
\end{equation}
Notice that by Lemma \ref{lem:1.12}.b $C^\infty (V)^G \underset{\mc O (\tilde V)^G}{\otimes} M$ is
a Fr\'echet $C^\infty (V)^G \underset{\mc O (\tilde V)^G}{\otimes} A$-bimodule, so the right 
hand side of \eqref{eq:1.15} is defined.

\begin{thm}\label{thm:1.10}
Let $A$ be unital and let $M$ be a finitely generated $A$-bimodule.
Assume that (i), (ii) and (iii) from Conditions \ref{cond} are fulfilled. 
Then \eqref{eq:1.15} induces a natural isomorphism of Fr\'echet $C^\infty (V)^G$-modules
\[
C^\infty (V)^G \underset{\mc O (\tilde V)^G}{\otimes} H_n (A,M) \to 
H_n \Big( C^\infty (V)^G \underset{\mc O (\tilde V)^G}{\otimes} A ,
C^\infty (V)^G \underset{\mc O (\tilde V)^G}{\otimes} M \Big) .
\]
\end{thm}
\begin{proof}
The algebra $A \otimes A^{op}$ is of finite type over $\mc O (\tilde V)^G \otimes 
\mc O (\tilde V)^G$. Hence \cite[Lemma 3]{KNS} applies to it, and yields a 
resolution $(A \otimes A^{op} \otimes F_*,d_*)$ of $A$ by finitely generated free
$A \otimes A^{op}$-modules. By definition
\begin{equation}\label{eq:1.16}
H_n (A,M) = H_n \big( (A \otimes A^{op}) \otimes F_* \underset{A \otimes A^{op}}{\otimes} M, 
d_* \otimes \mr{id} \big) = H_n (F_* \otimes M, d_*) .
\end{equation}
We note that by \cite[Proposition 2 and Corollary 1]{KNS} $H_n (A,M)$ is a finitely generated 
$\mc O (\tilde V)^G$-module, so applying $C^\infty (V)^G \underset{\mc O (\tilde V)^G}{\otimes}$
to it yields a Fr\'echet $C^\infty (V)^G$-module (Lemma \ref{lem:1.8}). We abbreviate 
\[
B = C^\infty (V)^G \underset{\mc O (\tilde V)^G}{\otimes} A \qquad \text{and} \qquad
N = C^\infty (V)^G \underset{\mc O (\tilde V)^G}{\otimes} M .
\] 
By the associativity of $\hat \otimes$ \cite[\S 2.1]{Mey1} there is a natural algebra isomorphism
\[
\big( C^\infty (V)^G \hat \otimes C^\infty (V)^G \big) \underset{\mc O (\tilde V)^G \otimes 
\mc O (\tilde V)^G}{\otimes} (A \otimes A^{op}) \cong B \hat \otimes B^{op} .
\]
Using that we put
\[
C_n = \big( C^\infty (V)^G \hat \otimes C^\infty (V)^G \big) \underset{\mc O (\tilde V)^G 
\otimes \mc O (\tilde V)^G}{\otimes} (A \otimes A^{op}) \otimes F_n \cong 
B \hat \otimes B^{op} \otimes F_n .
\]
Then Lemma \ref{lem:1.2} says that $(C_*,d_*)$ is a finitely generated free 
$B \hat{\otimes} B^{op}$-resolution of 
\begin{equation}\label{eq:1.20}
C^\infty (V)^G \underset{\mc O (\tilde V)^G}{\hat \otimes} A 
\underset{\mc O (\tilde V)^G}{\hat \otimes} C^\infty (V)^G .
\end{equation}
We warn that this resolution need not be split in the category of Fr\'echet spaces. By Lemma 
\ref{lem:1.12} the algebra \eqref{eq:1.20} is just $B$. 

Next we check all the conditions for \cite[Proposition 4.5]{Tay}.
The exactness of $(C_*,d_*)$ entails that im$(d_{n+1}) = \ker (d_n)$ is a closed subspace 
of  $C_n$, and in particular it is also Fr\'echet. The open mapping theorem for Fr\'echet spaces 
says that $d_n : C_n \to \mr{im}(d_n)$ is open, which in the terminology of \cite[\S 4]{Tay} means
that it is a topological homomorphism. By Lemmas \ref{lem:1.12} and \ref{lem:1.8},
\[
N \cong C^\infty (V)^G \underset{\mc O (\tilde V)^G}{\hat \otimes} M 
\underset{\mc O (\tilde V)^G}{\hat \otimes} C^\infty (V)^G 
\]
is a Fr\'echet space. Further, by Lemma \ref{lem:1.8} for $\mc O (\tilde V)^G \otimes 
\mc O (\tilde V)^G$, $B$ and all the $C_n$ are nuclear Fr\'echet spaces. Now we can apply 
\cite[Proposition 4.5]{Tay}, which says that $H_n (B, N)$ can be computed as 
\begin{equation}\label{eq:1.17}
H_n \big( B \hat \otimes B^{op} \otimes F_* \underset{B \hat \otimes B^{op}}{\otimes} N , 
d_* \otimes \mr{id} \big) = H_n (F_* \otimes N , d_*) .
\end{equation}
By the exactness of $C^\infty (V)^G \otimes_{\mc O (\tilde V)^G}$ from Theorem \ref{thm:1.1}, 
there are natural isomorphisms of $C^\infty (V)^G$-modules
\[
H_n (F_* \otimes N,d_*) \cong H_n \big( F_* \otimes C^\infty (V)^G 
\underset{\mc O (\tilde V)^G}{\otimes} M , d_* \big) \cong C^\infty (V)^G 
\underset{\mc O (\tilde V)^G}{\otimes} H_n (F_* \otimes M ,d_*) .
\]
Combine that with \eqref{eq:1.16} and \eqref{eq:1.17}. The resulting isomorphism shows that
$H_n ( B,N )$ is Hausdorff. In its construction as
\[
H_n (F_* \otimes N, d_*) = \ker (d_n) / \mr{im}(d_{n+1}) ,
\]
ker$(d_n)$ is closed by the continuity of $d_n$. By Hausdorffness, the image of $d_{n+1}$ must
be closed as well, which implies that the quotient $H_n (F_* \otimes N, d_*)$ is Fr\'echet.
\end{proof}

\section{Modules consisting of differential forms}

We preserve the setting of the previous paragraph.
To make good use of Theorem \ref{thm:1.10} we will make both 
\[
C^\infty (V)^G \underset{\mc O (\tilde V)^G}{\otimes} HH_n (A) \qquad \text{and} \qquad 
HH_n \Big( C^\infty (V)^G \underset{\mc O (\tilde V)^G}{\otimes} A \Big) 
\]
more explicit in some relevant classes of examples. As we are dealing with algebraic
tensor products, this involves checking that some modules are finitely generated. There have
been ample investigations of the structure of $C^\infty (V)^G$, starting with \cite{Sch}.
On the other hand, $C^\infty (V)$ has hardly been studied as $C^\infty (V)^G$-module.

Let $\pi$ be a representation of $G$ on a finite dimensional real vector 
space $W$. By classical results of Noether, see for instance \cite[\S 13.3]{Eis}, the ring 
of real valued polynomial functions $S(W^*)$ on $W$ is finitely generated as module over $S(W^*)^G$.

\begin{thm}\label{thm:1.5}
Let $G$ be a finite group.
\enuma{
\item $C^\infty (W)$ is generated as $C^\infty (W)^G$-module by a finite subset of $S(W^*)$.
\item Let $V$ be a smooth manifold with a smooth $G$-action. 
Then $C^\infty (V)$ is finitely generated as $C^\infty (V)^G$-module.
}
\end{thm}
\begin{proof}
(a) This is contained in \cite[Lemme III.1.4.1]{Poe}, but in disguise. Namely, it is stated there
that, for any finite dimensional real $G$-representation $(\pi',W')$,
\[
C^\infty_G (W,W') = \{ f \in C^\infty (W,W') : f (\pi (g) w) = \pi' (g) f(w) \,
\forall g \in G, w \in W \}
\]
is a finitely generated $C^\infty (W)^G$-module. We claim that, for $W' = \C [G]$ the left regular 
representation, there is an isomorphism of $C^\infty (W)^G$-modules
\begin{equation}\label{eq:1.6}
\begin{array}{ccc}
C^\infty (W) & \longleftrightarrow & C_G^\infty (W, \C [G]) , \\
f & \mapsto & [w \mapsto \sum_{g \in G} f (\pi (g^{-1}) w) g ]\\
\phi_1 & \text{\reflectbox{$\mapsto$}} & \phi = \sum_{g \in G} \phi_g g 
\end{array}.
\end{equation}
Indeed, the equivariance condition $\phi (\pi (g)w) = g \phi (w)$ means precisely that 
\[
\phi_g (w) = \phi_1 (\pi (g^{-1}) w) \text{ for all } w \in W. 
\]
Hence the two maps in \eqref{eq:1.6} are mutually inverse.
The proof of \cite[Lemme III.1.4.1]{Poe} uses only polynomial functions on $W \otimes W'^{*}$ as
generators, so via the isomorphism \eqref{eq:1.6} we can conclude that $C^\infty (W)$ is generated
by a finite subset of $S(W^*)$. In fact any set that generates $S(W^*)$ as $S(W^*)^G$-module will do.\\
(b) By \cite[Theorem 6.1]{Mos}, $V$ can be embedded $G$-equivariantly as a closed sub\-ma\-nifold 
in a space $W$ as in part (a). Thus we may and do regard $V$ as a subspace of $W$.
With part (a) we choose a finite set of generators $\{ f_i \}_i$ for $C^\infty (W)$ as 
$C^\infty (W)^G$-module. According to \cite[Th\'eor\`eme IX.4.3]{Tou}, the restriction map 
\[
C^\infty (W) \to C^\infty (V) : f \mapsto f |_V 
\]
is surjective. Hence the functions $f_i |_V$ generate $C^\infty (V)$ as $C^\infty (V)^G$-module. 
\end{proof}

In the algebraic setting, a theorem of Serre says that $\Omega^n (\tilde V)$ is
finitely generated as $\mc O (\tilde V)$-module, and hence also as $\mc O (\tilde V)^G$-module.
Similarly, the smooth Serre--Swan theorem says that $\Omega^n_{sm}(V)$ is finitely generated as 
$C^\infty (V)$-module, for any $n \in \Z_{\geq 0}$. This holds for any smooth manifold $V$,
compact or not \cite{Mor}.
By Theorem \ref{thm:1.5} $\Omega^n_{sm}(V)$ also finitely generated as $C^\infty (V)^G$-module.

In view of the Hochschild--Kostant--Rosenberg theorem \cite[Theorem 3.4.4]{Lod}, the Hochschild 
homology of finite type algebras will involve differential forms on varieties related to $\tilde V$. 
We will study such modules in a setting that starts with (i) and (ii) from Conditions \ref{cond}.
We assume that an embedding $\imath : \tilde Y_1 \to \tilde V$ is given, such that
\begin{itemize}
\item the image of $\imath$ is closed in $\tilde V$ and $\imath : \tilde Y_1 \to 
\imath(\tilde Y_1)$ is an isomorphism of affine algebraic varieties,
\item $Y_1 := \imath^{-1} (V)$ is a real analytic Zariski-dense submanifold of $\tilde Y_1$ and 
$\imath|_{Y_1} : Y_1 \to \imath (Y_1)$ is a diffeomorphism.
\end{itemize}
Thus $\imath$ induces algebra homomorphisms 
\[
\imath^* : C^\infty (V) \to C^\infty (Y_1) \quad \text{and} \quad
\imath^* : \mc O (\tilde V) \to \mc O (\tilde Y_1).
\]
Let $\tilde Y$ be a finite disjoint union of complex affine varieties $\tilde Y_j \; (j \in J)$, 
not necessarily of the same dimension, each of which has the same properties as those of 
$\tilde Y_1$ just listed. Let $Y$ be the disjoint union of the $Y_j$.

The above setup is used to study Schwartz algebras of reductive $p$-adic groups \cite[\S 3.1]{Sol2}.
However, let us point out that the standard and most instructive case of the upcoming results is
simply $\tilde Y = \tilde V, Y = V$.

\begin{lem}\label{lem:1.4}
With the above assumptions, let $C^\infty (V)^G$ act on $\Omega^n_{sm} (Y)$ via $\imath^*$. 
\enuma{
\item $\Omega^n (\tilde Y)$ is finitely generated as $\mc O (\tilde V )^G$-module.
\item $\Omega^n_{sm} (Y)$ is generated as $C^\infty (V)^G$-module by a finite subset of 
$\Omega^n (\tilde Y)$. 
}
\end{lem}
\begin{proof}
(a) By assumption $\imath( \tilde Y)$ is closed in $\tilde V$, so the restriction map
$\mc O (\tilde V) \to \mc O (\imath(\tilde Y))$ is surjective. As $\imath\big|_{\tilde Y}$ is
an isomorphism $\imath^* : \mc O (\tilde V) \to \mc O (\tilde Y)$ is surjective.
In particular $\Omega^n (\tilde Y)$ is a finitely generated module, over $\mc O (\tilde V)$
as well as over $\mc O (\tilde Y)$. Since 
$\mc O (\tilde V)$ is the integral closure of $\mc O (\tilde V)^G$ in the quotient field
of $\mc O (\tilde V)$, it has finite rank over $\mc O (\tilde V)^G$ \cite[Proposition 13.14]{Eis}.
Hence $\Omega^n (\tilde Y)$ is also finitely generated as $\mc O (\tilde V)^G$-module. \\
(b) By the smooth Serre--Swan theorem, $\Omega^n_{sm} (Y_j)$ is finitely generated over 
$C^\infty (Y_j)$. As $\imath (Y_j)$ is a closed submanifold of $V$, the restriction map 
$C^\infty (V) \to C^\infty (\imath (Y_j))$ is surjective \cite[Th\'eor\`eme IX.4.3]{Tou}. 
Since $\imath \big|_{Y_j}$ is a diffeomorphism, also
\begin{equation}\label{eq:1.5}
\imath^* : C^\infty (V) \to C^\infty (Y_j) \text{ is surjective.} 
\end{equation} 
In particular $\Omega^n_{sm} (Y_j)$ is a finitely generated $C^\infty (V)$-module, and so is
$\Omega^n_{sm} (Y) = \bigoplus_{j \in J} \Omega^n (Y_j)$. 
From the definition of the module structures we see that the tensor products 
\[
C^\infty (V)^G \otimes \Omega^n (\tilde Y) ,\quad C^\infty (V)^G \otimes \mc O (\tilde Y) \otimes
\Omega^n (\tilde Y) ,\quad  C^\infty (V)^G \otimes \mc O (\tilde V) \otimes \Omega^n (\tilde Y) .
\]
have the same image in $\Omega^n_{sm}(Y)$, under the natural action maps. By Theorem \ref{thm:1.5}.b
and \eqref{eq:1.5} the last one has the same image as
\[
C^\infty (V) \otimes \Omega^n (\tilde Y) \quad \text{and} \quad
C^\infty (Y) \otimes \Omega^n (\tilde Y) .
\]
The latter equals $\Omega^n_{sm}(Y)$, so $\Omega^n (\tilde Y)$ generates $\Omega^n_{sm} (Y)$ as
$C^\infty (V)^G$-module. By part (a) that can be achieved with a finite subset of 
$\Omega^n (\tilde Y)$.  
\end{proof}

Consider a $\mc O (\tilde V)^G$-submodule $M$ of $\Omega^n (\tilde Y)$, where the
action goes via $\imath^*$. Although it might seem obvious that 
$C^\infty (V)^G \underset{\mc O (\tilde V)^G}{\otimes} M$
embeds in $\Omega^n_{sm}(Y)$, that is actually about as difficult as Theorem \ref{thm:1.1}.

\begin{prop}\label{prop:1.6}
Assume that (i), (ii) and (iii) from Conditions \ref{cond} hold and let $M, Y$ and 
$\tilde Y$ be as above. The natural homomorphism of Fr\'echet $C^\infty (V)^G$-modules
$C^\infty (V)^G \underset{\mc O (\tilde V)^G}{\otimes} M \to \Omega^n_{sm}(Y)$
is injective.
\end{prop}
\begin{proof}
By Theorem \ref{thm:1.1} the natural map
\[
C^\infty (V)^G \underset{\mc O (\tilde V)^G}{\otimes} M \to
C^\infty (V)^G \underset{\mc O (\tilde V)^G}{\otimes} \Omega^n (\tilde Y)
\]
is injective. Therefore we may assume that $M = \Omega^n (\tilde Y)$. Then the statement
factors naturally as a direct sum indexed by $j \in J$. It suffices to consider one such direct
summand, say
\begin{equation}\label{eq:1.7}
C^\infty (V)^G \underset{\mc O (\tilde V)^G}{\otimes} \Omega^n (\tilde Y_1) \to
\Omega^n_{sm} (Y_1) .
\end{equation}
The formal completion of $\Omega^n_{sm} (Y_1)$ as $C^\infty (V)^G$-module at $Gv \in V /G$ is
\[
\bigoplus_{y \in \imath^{-1}(Gv)} FP_v^{G_v} \underset{C^\infty (V)^G}{\otimes} 
\widehat{C^\infty (Y_1)}_y \underset{\R}{\otimes} \bigwedge\nolimits^n (T_y (Y_1)^*) = 
\bigoplus_{y \in \imath^{-1}(Gv)} \widehat{C^\infty (Y_1)}_y \underset{\R}{\otimes} 
\bigwedge\nolimits^n (T_y (Y_1)^*).
\]
Using assumption (iii) we can also compute the formal completion of the left hand side of
\eqref{eq:1.7}:
\begin{align*}
\Big( C^\infty (V)^G \underset{\mc O (\tilde V)^G}{\otimes} \Omega^n (\tilde Y_1) 
\Big)^\wedge_{Gv} & = \bigoplus_{y \in \imath^{-1}(Gv)} FP_v^{G_v} 
\underset{\widetilde{FP}_v^{G_v}}{\otimes} \widehat{\mc O (\tilde Y_1)}_y 
\underset{\C}{\otimes} \bigwedge\nolimits^n \big( T_y (\tilde Y_1)^* \big) \\
& = \bigoplus_{y \in \imath^{-1}(Gv)} \widehat{\mc O (\tilde Y_1)}_y 
\underset{\C}{\otimes} \bigwedge\nolimits^n \big( T_y (\tilde Y_1)^* \big) .
\end{align*}
Assumption (iii) and the construction of $Y_1$ imply that 
$T_y (\tilde Y_1) = T_y (Y_1) \otimes_\R \C$. From that and the above we see that the map
\[
\big( C^\infty (V)^G \underset{\mc O (\tilde V)^G}{\otimes} \Omega^n (\tilde Y_1) 
\big)^\wedge_{Gv} \to \widehat{\Omega^n_{sm} (Y_1)}_{Gv} 
\]
induced by \eqref{eq:1.7} is injective. Now the same argument as for $\mu$ in the proof of 
Theorem \ref{thm:1.1} shows that \eqref{eq:1.7} is injective.
\end{proof}

Describing the image of the map from Proposition \ref{prop:1.6} is another issue.
One would like to think of it as some closure of $M$ in $\Omega^n_{sm}(Y)$,
but in general it is not clear whether the image is closed. To overcome that, we 
specialize to submodules of $\Omega^n (\tilde Y)$ that are direct summands.
Let $p$ be an idempotent in the ring of continuous $C^\infty (V)^G$-linear endomorphisms
of $\Omega^n_{sm} (Y)$, such that $p$ stabilizes $\Omega^n (\tilde Y)$. Then 
\begin{equation}\label{eq:1.8}
\Omega^n_{sm} (Y) = p \Omega^n_{sm} (Y) \oplus (1-p) \Omega^n_{sm} (Y) ,
\end{equation}
so $p \Omega^n_{sm} (Y)$ is a closed $C^\infty (V)^G$-submodule of $\Omega^n_{sm} (Y)$. 
Similarly $p \Omega^n (\tilde Y)$ is a $\mc O (\tilde V )^G$-submodule and a
direct summand of $\Omega^n (\tilde Y)$. 

\begin{lem}\label{lem:1.9}
Assume (i), (ii) and (iii) from Conditions \ref{cond}. The natural map
\[
\mu : C^\infty (V)^G \underset{\mc O (\tilde V)^G}{\otimes} p \Omega^n (\tilde Y)
\to p \Omega^n_{sm} (Y) 
\]
is an isomorphism of Fr\'echet $C^\infty (V)^G$-modules.
\end{lem}
\begin{proof}
By construction the image of $\mu$ is contained in $p \Omega^n_{sm} (Y)$ and we know from 
Proposition \ref{prop:1.6} that $\mu$ is injective. By Lemma \ref{lem:1.4}.b any 
$m \in p \Omega^n_{sm} (Y)$ can be written as a finite sum $m = \sum_i f_i \omega_i$ 
with $f_i \in C^\infty (V)^G$ and $\omega_i \in \Omega^n (\tilde Y)$. We compute
\[
m = p(m) = p \big( \sum\nolimits_i f_i \omega_i \big) = \sum\nolimits_i f_i p(\omega_i) 
\in \mu \Big( C^\infty (V)^G \underset{\mc O (\tilde V)^G}{\otimes} p \Omega^n (\tilde Y) \Big) . 
\]
In other words, $\mu$ is surjective. In view of Proposition \ref{prop:1.6}, $\mu$ is
a continuous bijection between Fr\'echet spaces. Now the open mapping theorem says 
that it is a homeomorphism.
\end{proof}

\section{Special cases}
\label{sec:ex}

Consider the algebra $C^\infty (V)$ with $V, \tilde V$ as in Conditions \ref{cond},
for the moment without any group action. By Theorem \ref{thm:1.10}  
\begin{equation}\label{eq:2.5}
HH_n (C^\infty (V)) = HH_n \Big( C^\infty (V) \underset{\mc O (\tilde V)}{\otimes} 
\mc O (\tilde V) \Big) \cong 
C^\infty (V) \underset{\mc O (\tilde V)}{\otimes} HH_n \big( \mc O (\tilde V) \big) .
\end{equation}
Here we may remove the singular locus of $\tilde V$, because it does not meet $V$.
Then $\tilde V$ is nonsingular, so we can invoke the Hochschild--Kostant--Rosenberg theorem 
\cite[Theorem 3.4.4]{Lod}. Next we apply Lemma \ref{lem:1.9} to the right hand side
of \eqref{eq:2.5} and we find natural isomorphisms 
\[
HH_n (C^\infty (V)) \cong C^\infty (V) \underset{\mc O (\tilde V)}{\otimes} 
\Omega^n (\tilde V) \cong \Omega^n_{sm} (V) .
\]
In this way we recover Connes' version of the Hochschild--Kostant--Rosenberg theorem
\cite{Con}, for the Hochschild homology of the Fr\'echet algebra of smooth functions 
on a real analytic manifold $V$. Because of the techniques that we used, our proof only
applies when $V$ can be embedded in a complex affine variety $\tilde V$ such that
$T_v (\tilde V) = T_v (V) \otimes_\R \C$ for all $v \in V$.

Interesting examples arise from imposing conditions in terms of an affine subvariety 
$\tilde W \subset \tilde V$. For instance, let $k \in \N$ and consider the unital 
finite type $\mc O (\tilde V)$-algebra
\[
A = \{ \matje{a}{b}{c}{d} \in M_2 (\C) \otimes \mc O (\tilde V) : 
c \text{ vanishes to the order } k \text{ on } \tilde W \} .
\]
From Lemma \ref{lem:1.9} one can deduce that
\[
C^\infty (V) \underset{\mc O (\tilde V)}{\otimes} A = \{ \matje{a}{b}{c}{d} \in M_2 (\C) 
\otimes C^\infty (V) : c \text{ vanishes to the order } k \text{ on } \tilde W \cap V \} .
\]
In principle $HH_* (A)$ can be computed with the techniques from \cite{KNS}. Thus Theorem 
\ref{thm:1.10} provides an approach to determine 
$HH_* (C^\infty (V) \underset{\mc O (\tilde V)}{\otimes} A)$.

Next we consider the crossed product algebra $\mc O (\tilde V) \rtimes G$, 
where $G$ is a finite group acting on $V$ and on $\tilde V$. Its Hochschild homology has
been determined in \cite[Theorem 2.11]{Nis}:
\begin{equation}\label{eq:2.6}
HH_n (\mc O (\tilde V) \rtimes G) \cong \bigoplus\nolimits_{g \in \langle G \rangle}
\Omega^n (\tilde V^g)^{Z_G (g)} ,
\end{equation}
where $\langle G \rangle$ is a set of representatives for the conjugacy classes in $G$.
Similary, it is known from \cite[Proposition 6]{Bry} that
\begin{equation}\label{eq:2.7}
HH_n (C^\infty (V) \rtimes G) \cong \bigoplus\nolimits_{g \in \langle G \rangle}
\Omega^n_{sm} (V^g)^{Z_G (g)} .
\end{equation}
As $\mc O (\tilde V)$ has finite rank over $\mc O (\tilde V)^G$, $\mc O (\tilde V) \rtimes G$
is a finite type $\mc O (\tilde V)^G$-algebra. By Lemma \ref{lem:1.9}
\[
C^\infty (V)^G \underset{\mc O (\tilde V)^G}{\otimes} \mc O (\tilde V) \rtimes G 
\cong C^\infty (V) \rtimes G .
\]
Now Theorem \ref{thm:1.10} says that 
\begin{equation}\label{eq:2.8}
\begin{aligned}
HH_n (C^\infty (V) \rtimes G) & \cong C^\infty (V)^G \underset{\mc O (\tilde V)^G}{\otimes} 
HH_n (\mc O (\tilde V) \rtimes G) \\
& \cong C^\infty (V)^G \underset{\mc O (\tilde V)^G}{\otimes} 
\bigoplus\nolimits_{g \in \langle G \rangle} \Omega^n (\tilde V^g)^{Z_G (g)} .
\end{aligned}
\end{equation}
By Lemma \ref{lem:1.9}, with $p |_{\Omega^n_{sm} (V^g)}$ the projection to 
$Z_G (g)$-invariants, the right hand side of \eqref{eq:2.8} is isomorphic to
\[
\bigoplus\nolimits_{g \in \langle G \rangle} \Omega^n_{sm} (V^g)^{Z_G (g)} .
\]
Thus our results agree with the earlier findings from \cite{Bry,Nis}.

A more challenging class of examples arises as follows. Suppose that $G$ acts on 
$M_n (\C) \otimes \mc O (\tilde V) = M_n (\mc O (\tilde V))$ by 
\[
g \cdot f = u_g (f \circ g^{-1}) u_g^{-1},
\]
where $u_g \in M_n (\mc O (\tilde V))^\times$ and $f$ is regarded as a map from $\tilde V$
to $M_n (\C)$. Then
\begin{equation}\label{eq:2.9}
A = M_n (\mc O (\tilde V))^G
\end{equation}
is a finite type $\mc O (\tilde V)^G$-algebra. Special cases of this construction are
$\mc O (\tilde V)^G$ (for $n=1$) and $\mc O (\tilde V) \rtimes G$, for $M_n (\C)
= \End (\C [G])$. As far as we are aware, there is no general formula for the Hochschild
homology of such algebras. By Lemma \ref{lem:1.9}
\begin{equation}\label{eq:2.10}
C^\infty (V)^G \underset{\mc O (\tilde V)^G}{\otimes} A \cong
M_n (C^\infty (V))^G .
\end{equation}
Algebras of the form \eqref{eq:2.9} and \eqref{eq:2.10} are relevant because they arise in
abundance from reductive $p$-adic groups, see for instance \cite{Sol2}.

\end{document}